\documentclass[a4paper,12pt,reqno]{amsart}
\usepackage[utf8]{inputenc}
\usepackage{amsaddr}
\usepackage{amssymb}
\usepackage{calrsfs}
\usepackage{graphicx,xcolor}
\usepackage[colorlinks = true,linkcolor = blue,urlcolor  = blue,citecolor = blue,anchorcolor = blue]{hyperref}
\usepackage{dsfont}
\usepackage{lipsum}
\usepackage{setspace}
\usepackage{caption}
\usepackage{float}
\usepackage{orcidlink}
\usepackage{bm}
\usepackage{cite}
\usepackage{booktabs}
\allowdisplaybreaks
\numberwithin{equation}{section}

\theoremstyle{definition}
\newtheorem{dfn}{Definition}[section]

\theoremstyle{plain}
\newtheorem{thm}[dfn]{Theorem}
\newtheorem{pro}[dfn]{Proposition}
\newtheorem{cor}[dfn]{Corollary}
\newtheorem{lmm}[dfn]{Lemma}
\newtheorem{rem}[dfn]{Remark}

\newcommand{\R}{\mathbb{R}}
\newcommand{\E}{\mathbb{E}}
\renewcommand{\P}{\mathbb{P}}
\newcommand{\Var}{\mathbb{V}\mathrm{ar}}
\newcommand{\Cov}{\mathbb{C}\mathrm{ov}}
\newcommand{\Cor}{\mathbb{C}\mathrm{or}}
\newcommand{\erf}{\mathrm{erf}}
\newcommand{\LN}{\mathcal{L}\!{og}\mathcal{N}\!{ormal}}

\newcommand{\1}{\mathds{1}}

\newcommand{\ceil}[1]{\left\lceil #1 \right\rceil}

\begin{document}
\thispagestyle{empty}
\begin{center}
	{\scshape\large Insights into Tail-Based and Order Statistics}
\end{center}
\vspace*{2cm}
\begin{figure}[H]
	\begin{minipage}[b]{.3\linewidth}
		\includegraphics[scale=1.2]{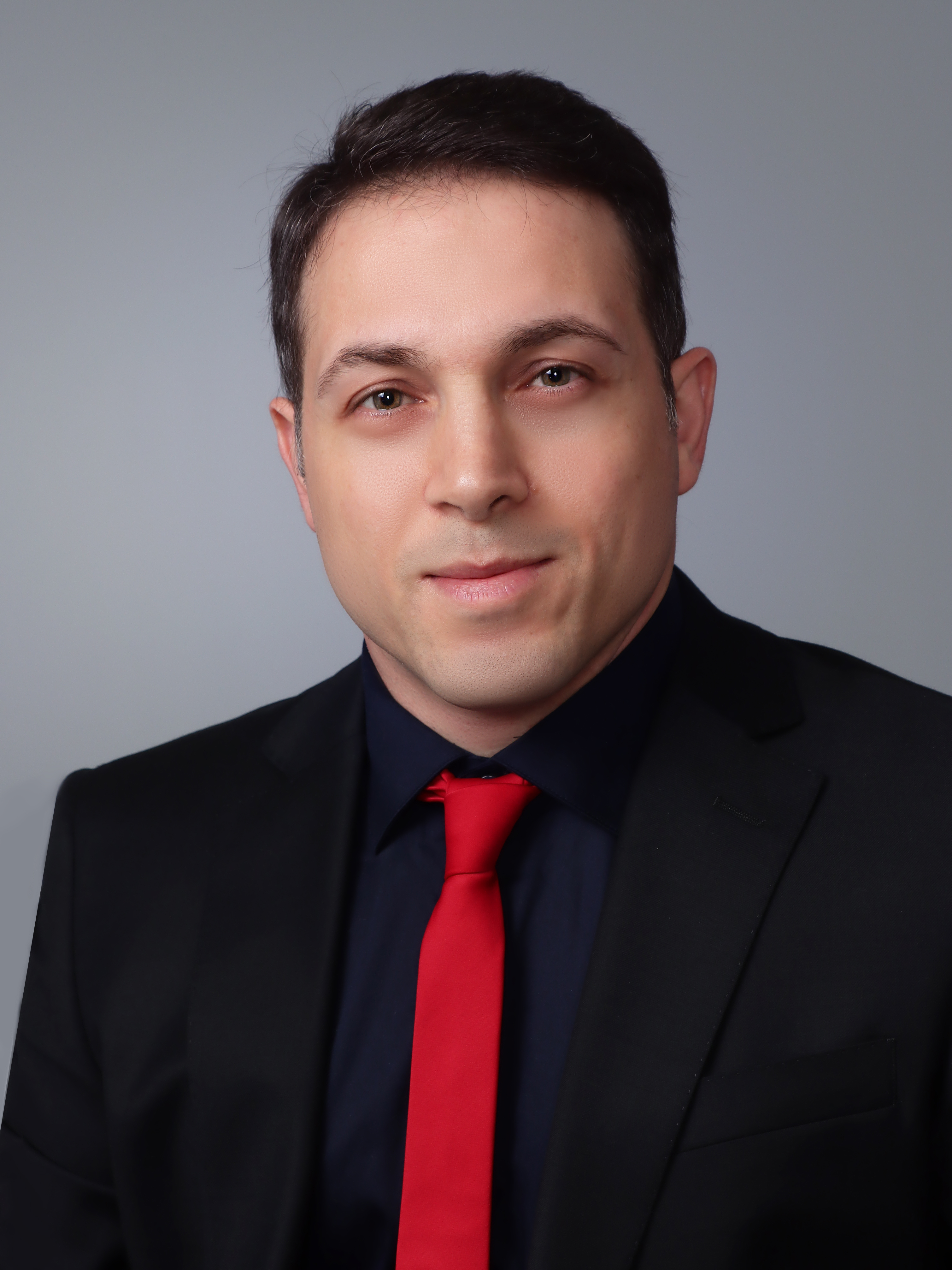}\\
		\vspace{1cm}
	\end{minipage}\hfill
	\begin{minipage}[b]{.7\linewidth}
		{\large\bfseries Hamidreza Maleki Almani}\ {\Large\orcidlink{0000-0002-3071-4982}}\\
		%\ \\
		ORCID: \href{https://orcid.org/0000-0002-3071-4982}{https://orcid.org/0000-0002-3071-4982}\\
		Web: \href{https://www.uwasa.fi/en/person/2169161}{https://www.uwasa.fi/en/person/2169161}\\
		\ \\
        This article is an independent work by the author. He is a {\it Postdoctoral Researcher} in the Department of Mathematics and Statistics and the Department of Energy Technology at the University of Vaasa, Finland. He independently conducted and completed all aspects of this study.\\
        \ \\
        November 6, 2025\\
        Vaasa, FINLAND
	\end{minipage}
\end{figure}
\newpage
\setcounter{page}{1}
\title[Insights into Tail-Based and Order Statistics]{}
\vspace*{-1cm}
\begin{center}
	{\Large\bfseries Insights into Tail-Based and Order Statistics
	}
\end{center}
\date{\today}
\vspace*{-1.5cm}
\author[Almani]{Hamidreza Maleki Almani {\Large\orcidlink{0000-0002-3071-4982}}}
\address{Department of Mathematics and Statistics, School of Technology and Innovations, University of Vaasa, P.O. Box 700, FIN-65101 Vaasa, FINLAND}
\address{ORCID: \href{https://orcid.org/0000-0002-3071-4982}{https://orcid.org/0000-0002-3071-4982}}
\email{hmaleki@uwasa.fi}

%%%%%%%%%%%%%%%%%%%%%%%%%%%%%%%%%%%%%%%%%%%%%%%%%%%%%%%%%%%%%%%%%%%%%%%%%%%%%%%

\begin{abstract}
Heavy-tailed phenomena appear across diverse domains—from wealth and firm sizes in economics to network traffic, biological systems, and physical processes—characterized by the disproportionate influence of extreme values. These distributions challenge classical statistical models, as their tails decay too slowly for conventional approximations to hold. Among their key descriptive measures are quantile contributions, which quantify the proportion of a total quantity (such as income, energy, or risk) attributed to observations above a given quantile threshold.
This paper presents a theoretical study of the quantile contribution statistic and its relationship with order statistics. We derive a closed-form expression for the joint cumulative distribution function (CDF) of order statistics and, based on it, obtain an explicit CDF for quantile contributions applicable to small samples. We then investigate the asymptotic behavior of these contributions as the sample size increases, establishing the asymptotic normality of the numerator and characterizing the limiting distribution of the quantile contribution. Finally, simulation studies illustrate the convergence properties and empirical accuracy of the theoretical results, providing a foundation for applying quantile contributions in the analysis of heavy-tailed data.
\end{abstract}



\maketitle

%%%%%%%%%%%%%%%%%%%%%%%%%%%%%%%%%%%%%%%%%%%%%%%%%%%%%%%%%%%%%%%%%%%%%%%%%%%%%%%
\section{Introduction}\label{sec:introduction}
In 1906, Pareto, in his first well-known work \cite{pareto1964cours}, showed that approximately 80\% of the land in the Kingdom of Italy was owned by only 20\% of the population at that time. This became known as Pareto's 80/20 principle.
Sturgeon's publications in the 1950s \cite{sturgeon1956Claustrophile,sturgeon1957hand,sturgeon1957offhand,sturgeon1958law} highlighted the observation that the majority of everything is of low quality. However, the prevalence of low-quality content across all genres disproves the notion that any single genre is inherently inferior. This idea is now known as Sturgeon's adage: "{\it Ninety percent of everything is crud!}"
Computer programmers are familiar with this in another form \cite{bentley1985programmimg,mantle2012managing}: in computer programming and software engineering, the Ninety–Ninety Rule is a humorous aphorism that states, “{\it The first 90\% of the code accounts for the first 90\% of the development time, and the remaining 10\% of the code accounts for the other 90\% of the development time!}” This adds up to 180\%, making a wry allusion to the notorious tendency of software development projects to significantly overrun their schedules.
In global health care, as a seriouse issue,  the 10/90 gap is a term adopted by the Global Forum for Health Research to highlight the finding by the Commission on Health Research for Development in 1990 that less than 10\% of worldwide resources devoted to health research were allocated to developing countries—where over 90\% of all preventable deaths worldwide occur (see \cite{vidyasagar2006global,doyal2004gender,global2011report}). This disparity is a major concern of the World Health Organization (WHO) \cite{WHO2004report,davey200410}.
This is observed even more sharply in internet culture \cite{Charles2006onepercent,van20141}. The 1\% rule is a general rule of thumb regarding participation in an online community, stating that only 1\% of a website’s users actively create new content, while the other 99\% simply lurk.

The observations mentioned above relate to a deeper fact beyond mere statistical inference. Informally, to estimate the probability of an event, it is often sufficient to focus on the concentration region of its distribution—provided we have a large enough sample and the distribution’s tails “vanish rapidly enough.” However, this assumption does not hold if the tails are thicker than negligible. In such cases, infrequent events have a significant probability, meaning that the usual “well-behaved” statistical models fail to accurately represent them. This is when the tails of the distribution must be taken into account, leading to what are called heavy-tailed processes. The historical evolution of heavy-tailed phenomena, some of which we have mentioned, reveals the following setup:
\begin{enumerate}
	\item[$\bullet$] {\it Vanishing rapidly enough} means a negligible tail, typically vanishing exponentially,
	\item[$\bullet$] {\it Well-behaved} also refers to distributions with exponentially vanishing tails.
\end{enumerate}
So, a distribution $F$ is heavy-tailed \cite{rolski2009stochastic,foss2011introduction} if $1-F(x)=\P[X>x]\gg e^{-sx}$ for $x\to\infty$ and $s>0$, i.e.,
$$
\lim_{x\to\infty} e^{sx}(1-F(x))\,=\,\infty.
$$ 
Three well-known sub-classes of the heavy-tailed distributions are
\begin{enumerate}
	\item [(i)] Fat-tailed distributions \cite{mikosch1999regular,mandelbrot2010mis} with index $0<\alpha<2$ that 
	$$1-F(x)\sim x^{-\alpha}\text{ for }x\to\infty,$$
	\item [(ii)] Long-tailed distribution \cite{asmussen2003steady} that for all $t>0$ we have
	$$1-F(x+t)\sim 1-F(x)\text{ for }x\to\infty,$$
	\item [(iii)] Subexponential distributions \cite{embrechts2013modelling,chistyakov1964theorem} that for all independent processes $X_1,\ldots,X_n\sim F$ we have 
	$$\P[X_1+\cdots+X_n>x]\sim\P[\max(X_1,\ldots,X_n)>x]\text{ for }x\to\infty.$$
\end{enumerate}

Heavy-tailed distributions are crucial in numerous scientific fields due to their ability to model rare, high-impact events and skewed distributions. In economy, finance, and business, they capture extreme asset returns \cite{mandelbrot1963variation,fama1963mandelbrot}, volatility clustering \cite{cont2001empirical}, and market shocks \cite{sornette2009stock}, enhancing risk modeling and forecasting. Wealth and firm size distributions often follow power laws, aiding economic analysis \cite{gabaix2009power,axtell2001zipf}. These models also inform business strategies in sales, resource allocation, and resilience to demand shocks \cite{taleb2010black,embrechts2013modelling,lux1999scaling}. In computer science, heavy-tailed patterns appear in internet traffic \cite{leland2002self}, file sizes \cite{downey2001structural}, and server loads, affecting network protocols and performance \cite{crovella2002self,paxson1995wide}. They also underpin job scheduling in distributed systems \cite{harchol2003size,LIN2007856} and anomaly detection in cybersecurity \cite{barford2002signal}.

In physics and engineering, heavy-tailed distributions describe anomalous diffusion \cite{shlesinger1993strange}, turbulent transport \cite{metzler2000random}, and structural failure in materials \cite{bavzant2004scaling,castillo2001general}. They are used in modeling impulsive noise and signal degradation in communication systems \cite{nikias1995signal,middleton2002non}, as well as robust signal processing under uncertainty.% (Míguez \& Djuric, 2006). 
In biology and health sciences, these distributions explain superspreading in epidemics \cite{lloyd2005superspreading,endo2020estimating}, scale-free gene and protein networks \cite{barabasi2004network,jeong2001lethality}, and variability in neural dynamics \cite{buzsaki2014log}. They also capture skewed healthcare metrics such as drug response and hospital stays \cite{clauset2009power}. Across disciplines, heavy-tailed distributions support more realistic, data-driven modeling of complex systems, improving prediction, design, and decision-making.

The most important statistics of a heavy-tailed distribution are its quantiles. This is because we aim to identify a precise split of the distribution into two parts: the head and the tail. Specifically, we look for the point below which a considerable percentage $p$ of the population falls (the pth quantile), while the accumulated value of the remaining $(1-p)$ percent above that point constitutes a significant portion of the total value in the population. This measure is known as the {\it quantile contribution} \cite{taleb2015super}. It refers to the proportion of a total quantity—such as income, risk, energy, or emissions—attributed to elements above (or sometimes below) a certain quantile threshold within a statistical distribution. In simple terms, it shows how much of a total amount is accounted for by a particular subset of units ranked by size (e.g., income, energy, or emissions).
The ``natural'' estimator for the quantile contribution is calculated as the ratio of the sum of values above the exceedance threshold (the value above a specific quantile) to the total sum. That is
\begin{equation}\label{eq:Lambda_n}
	\Lambda_n(p) = \frac{\sum_{j\in\mathcal{J}_n(p)}X_j}{\sum_{j=1}^nX_j},
\end{equation}
where $p\in[0,1]$ is a constant number, $X_j, j=1,\ldots,n$ are independently identically distributed (i.i.d), and
\begin{equation}
	\mathcal{J}_n(p) = \{ j\,|\, X_j\in 100p\%\,\text{largest observations among}\, X_1,\ldots,X_n\},
\end{equation}
we note that for all $p\in [0,1], n\ge 1$ we have $|\lambda_n(p)|\le 1$.

In this article, we study the connection between quantile contributions and order statistics, focusing on their distributions and convergence. In Section \ref{sec:OrderedAndTailed}, we derive a closed-form expression for the joint cumulative distribution function (CDF) of order statistics. Building on this, Section \ref{sec:exact.dist} presents an explicit form of the CDF for quantile contributions, applicable to a small number of variables. Section \ref{sec:convergence} explores the convergence of quantile contributions as the number of variables grows large. Section \ref{sect: Asymp.Dist.Num} presents the asymptotic normality of the numerator, and Section \ref{sect: Asymp.Dist.Lamb} applies this result to characterize the asymptotic distribution of quantile contributions for a large number of variables. Finally, in Section \ref{sec:simulate_callibrate}, we present simulations of important cases and cumulative errors to illustrate the empirical performance and accuracy of our results.
\section{Ordered and Tail-Based Statistics}\label{sec:OrderedAndTailed}
To investigate the convergence and distribution of the $\Lambda_n$ given in \eqref{eq:Lambda_n}, first we must consider its close relationship to the order statistics. For the random variables $X_1, X_2,\ldots, X_n$, the associated order statistics are the random variables $X_{(1)}^n,X_{(2)}^n,\ldots,X_{(n)}^n$ defined by ascending resorting of the $X_1, X_2,\ldots, X_n$. Then we have
\begin{equation}\label{eq:LOS}
	\Lambda_n(p) 
	= \frac{\sum_{i=\ceil{np}}^n X_{(i)}^n}{\sum_{i=1}^n X_i}
	= \frac{\sum_{i=\ceil{np}}^n X_{(i)}^n}{\sum_{i=1}^n X_{(i)}^n}.
\end{equation}
In other word, to investigate the probability distribution of $\Lambda_n$, it is sufficient to know the joint distribution of the order statistics $X_{(1)}^n,X_{(2)}^n,\ldots,X_{(n)}^n$.
The following proposition for the distribution of each $X_{(i)}^n$ is explained in \cite{cooke2014fat,reiss2012approximate,arnold2008first}. Here I just rewrite the proof with a quantitative formulation of its combinatorics.
\begin{pro}
	The probability distribution $F_{(i)}^n$ and density function $f_{(i)}^n$ of the order statistic $X_{(i)}^n$ are
	\begin{align}
		F_{(i)}^n(x) &= I\big(F(x);\,i,n-i+1\big)
		= \sum_{J=i}^n\binom{n}{J}\Big(F(x)\Big)^J\Big(1-F(x)\Big)^{n-J},\label{eq:F}\\
		f_{(i)}^n(x) &= \frac{f(x)}{B(i,n-i+1)}\Big(F(x)\Big)^{i-1}\Big(1-F(x)\Big)^{n-i},\label{eq:f}
	\end{align}
	where $F$ and $f$ are respectively the probability distribution and density function of the variable $X_1$, and $B,I$ are respectively the beta function and the regularized inclomplete beta function, i.e., for all $\mathfrak{Re}(p),\mathfrak{Re}(q)>0$
	\begin{gather*}
		B(p,q)=\int_0^1 t^{p-1}(1-t)^{q-1}\,dt,\\
		I(x;\,p,q)=\frac{1}{B(p,q)}\int_0^x t^{p-1}(1-t)^{q-1}\,dt.
	\end{gather*}
\end{pro}
\begin{proof}
	As $X_{(i)}^n$ is the $(i/n)$-quantile variable of $X_1,\ldots,X_n$, for all $x\in\R$
%	For $X=(x_1,\ldots ,x_n)\in\R^n$, then $X_{(i)}^n=x_r$ if and only if $(\frac{i}{n}\times 100)\%$ of $x_1,\ldots,x_n$ are less or equal to $x_r$, and $(\frac{n-i}{n}\times 100)\%$ of $x_1,\ldots,x_n$ are greater or equal to $x_r$. In other word, $X_{(i)}^n$ is the  the process $X$ that we denote with $Q_n(i/n)$. It is
%	\begin{equation}
%		X_{(i)}^n=Q_n(i/n)=x_r\iff i=\sum_{j=1}^n\1_{(-\infty,x_r]}(X_j).
%	\end{equation}
%	By the properties of quantile
	\begin{equation*}
		X_{(i)}^n=Q_n(i/n)\le x\iff i\le\sum_{j=1}^n\1_{(-\infty,x]}(X_j).
	\end{equation*}
	So,
	\begin{align*}
		F_{(j)}^n(x)&=\P[X_{(i)}^n\le x]\\
		&=\P\left[i\le\sum_{j=1}^n\1_{(-\infty,x]}(X_j)\right]\\
		&= \sum_{J=i}^n\P\left[J=\sum_{j=1}^n\1_{(-\infty,x]}(X_j)\right]\\
		&= \sum_{J=i}^n\binom{n}{J}\Big(\P[X_1\le x]\Big)^J\Big(\P[X_1>x]\Big)^{n-J}\\
		&=\sum_{J=i}^n\binom{n}{J}\Big(F(x)\Big)^J\Big(1-F(x)\Big)^{n-J}.
	\end{align*}
	Now, we note
	\begin{align*}
		&\sum_{J=i}^n\binom{n}{J}y^J(1-y)^{n-J}\\ 
		&= \frac{\displaystyle\int_0^y t^i(1-t)^{n-i+1}\,dt}{B(i,n-i+1)}\\ 
		&= I\big(y;\,i,n-i+1\big),
	\end{align*}
	and these prove \eqref{eq:F}. The proof of \eqref{eq:f} is straight forward as follows.
	\begin{align*}
		f_{(j)}^n(x)&=\frac{dF_{(j)}^n(x)}{dx}\\
		&=\frac{dI\big(F(x);\,i,n-i+1\big)}{dx}\\
		&=\frac{f(x)\big(F(x)\big)^{i-1}\big(1-F(x)\big)^{n-i}}{B(i,n-i+1)}.
	\end{align*} 
\end{proof}
The joint density function of the order statistics $X_{(1)}^n,X_{(2)}^n,\ldots,X_{(n)}^n$ is given by \cite{reiss2012approximate,arnold2012relations} as following theorem and corollary.
\begin{thm}\label{thm: density}
	Let $1\le k\le n$ and $0=r_0<r_1<\cdots<r_k<r_{k+1}=n+1$. If the random variables $X_1,X_2,\ldots,X_n$ are i.i.d with common absolutely continuous distribution $F$ and density function $f$, then the joint density function of $X_{(r_1)}^n,X_{(r_2)}^n,\ldots,X_{(r_k)}^n$ is
	\begin{equation}
		f_{(r_1,\ldots,r_k)}^n(x_1,\ldots,x_k) 
		= n!\left(\prod_{i=1}^k f(x_i)\right)
		\prod_{i=1}^{k+1}\frac{\big(F(x_i)-F(x_{i-1})\big)^{r_i-r_{i-1}-1}}{(r_i-r_{i-1}-1)!},
	\end{equation}
	if $x_1<x_2<\cdots<x_k$, and it is $0$ otherwise. Here $F(x_0)=0$ and $F(x_{k+1})=1$.
\end{thm}
\begin{cor}\label{cor: fnn}
	If the random variables $X_1,X_2,\ldots,X_n$ are i.i.d with common absolutely continuous distribution $F$ and density function $f$, then the joint density function of $X_{(1)}^n,X_{(2)}^n,\ldots,X_{(n)}^n$ is
	\begin{equation}
		f_{(1,\ldots,n)}^n(x_1,\ldots,x_n) 
		= n!\prod_{i=1}^n f(x_i),
	\end{equation}
	if $x_1<x_2<\cdots<x_n$, and it is $0$ otherwise.
\end{cor}
Next, we evaluate the cumulative distribution function of the order statistics. However, this requires some insight into the relationship between the Binomial distribution and the regularized incomplete Beta function, as presented in the following lemma.
\begin{lmm}\label{lmm: Ibeta}
For all positive integers $p,q\ge 1$, and all $a,b\in\R$
	\begin{align*}
	I_{a,b}(y;\,p,q)
	&:=\frac{1}{B(p,q)}\int_a^y (x-a)^{p-1}(b-x)^{q-1}\,dx\\
	&=\sum_{j=p}^{p+q-1}\binom{p+q-1}{j}(y-a)^{j}(b-y)^{p+q-1-j}\\
	&=\sum_{j=0}^{q-1}\binom{p+q-1}{j}(y-a)^{p+q-1-j}(b-y)^{j}.
	\end{align*} 
\end{lmm}
\begin{proof}
	By changing the variable $t=\frac{x-a}{b-a}$, we have
	\begin{equation*}
		I_{a,b}(y;p,q)=(b-a)^{p+q-1}I\left(\frac{y-a}{b-a};p,q\right),
	\end{equation*}
	and
	\begin{align*}
	I\left(u;p,q\right)&=\P[J\ge p]\\
	&=\sum_{j=p}^{p+q-1}\binom{p+q-1}{j}u^{j}(1-u)^{p+q-1-j}\\
	&=\sum_{j=0}^{q-1}\binom{p+q-1}{j}u^{p+q-1-j}(1-u)^{j},
	\end{align*} 
	where $J\sim\mathcal{B}inomial(u;p+q-1)$. Now, substituting $u=\frac{y-a}{b-a}$ proves the claim.
\end{proof}
\begin{thm}\label{thm: CDF_Ordered}
	Let $1\le k\le n$ and $0<r_1<\cdots<r_k<n+1$ are integers. If the random variables $X_1,X_2,\ldots,X_n$ are i.i.d with common absolutely continuous distribution $F$ and density function $f$, then the cumulative distribution function of $X_{(r_1)}^n,X_{(r_2)}^n,\ldots,X_{(r_k)}^n$ is
	\begin{align}
		&F_{(r_1,\ldots,r_k)}^n(x_1,\ldots,x_k)\notag\\ 
		&=\sum_{J_k=0}^{n-r_k}\quad\sum_{J_{k-1}=0}^{n-r_{k-1}-J_k}\quad\sum_{J_{k-2}=0}^{n-r_{k-2}-J_k-J_{k-1}}\cdots\sum_{J_1=0}^{n-r_1-\sum_{i=2}^kJ_i}\notag\\
		&\quad\binom{n}{J_0,J_1,\ldots,J_k}\prod_{i=0}^{k}\Big(F\left(x_{(i+1)}^k\right)-F\left(x_{(i)}^k\right)\Big)^{J_i},\notag\\
		&s.t.\quad\sum_{i=0}^k J_i = n,\label{eq: P0}
	\end{align}
	if $x_1<x_2<\cdots<x_k$, and it is $0$ otherwise. Here $F(x_0)=0$, and $F(x_{k+1})=1$. (Note: $J_0 = n - \sum_{i=1}^k J_i$)
\end{thm}
\begin{proof}[\bf Proof 1: Calculus]
	Applying the density function from Theorem \ref{thm: density}, for $y_1\le\cdots\le y_k$ we have
	\begin{align*}
		&F_{(r_1,\ldots,r_k)}^n(y_1,\ldots,y_k)\\
		&=n!\int_{-\infty}^{y_1}\int_{x_1}^{y_2}\cdots\int_{x_{k-2}}^{y_{k-1}}\int_{x_{k-1}}^{y_k}dx_k\cdots dx_1\\
		&\quad\times\left(\prod_{i=1}^k f(x_i)\right)
		\prod_{i=1}^{k+1}\frac{\big(F(x_i)-F(x_{i-1})\big)^{r_i-r_{i-1}-1}}{(r_i-r_{i-1}-1)!},
	\end{align*}
	where $r_0=0$ and $r_{k+1}=n+1$. By changing the variables $u_i=F(x_i),\, i=1,\ldots k+1$, we have
	\begin{align*}
		&=n!\int_{0}^{F(y_1)}\int_{u_1}^{F(y_2)}\cdots\int_{u_{k-2}}^{F(y_{k-1})}\int_{u_{k-1}}^{F(y_k)}
		\prod_{i=1}^{k+1}\frac{(u_i-u_{i-1})^{r_i-r_{i-1}-1}}{(r_i-r_{i-1}-1)!}\,du_k\cdots du_1\\
		&=n!\underbrace{\int_{0}^{F(y_1)}\int_{u_1}^{F(y_2)}\cdots\int_{u_{k-2}}^{F(y_{k-1})}
		\prod_{i=1}^{k-1}\frac{(u_i-u_{i-1})^{r_i-r_{i-1}-1}}{(r_i-r_{i-1}-1)!}\,du_{k-1}\cdots du_1}_{\pi_{k-1}}\\
		&\quad\times\int_{u_{k-1}}^{F(y_k)}\frac{(u_{k+1}-u_{k})^{r_{k+1}-r_k-1}(u_k-u_{k-1})^{r_k-r_{k-1}-1}}{(r_{k+1}-r_k-1)!(r_k-r_{k-1}-1)!}\, du_k\\
		&= n!\,\pi_{k-1}\int_{u_{k-1}}^{F(y_k)}\frac{(u_{k+1}-u_{k})^{r_{k+1}-r_k-1}(u_k-u_{k-1})^{r_k-r_{k-1}-1}}{B(r_{k+1}-r_k,r_k-r_{k-1}-1)(r_{k+1}-r_{k-1}-1)!}\, du_k.
	\end{align*}
	Taking $p_k=r_k-r_{k-1}, q_k=r_{k+1}-r_k$, then by Lemma \ref{lmm: Ibeta}
	\begin{align*}
		&= \frac{n!\,\pi_{k-1}}{(r_{k+1}-r_{k-1}-1)!}\,I_{u_{k-1},u_{k+1}}\Big(F(y_k);r_k-r_{k-1},r_{k+1}-r_k\Big)\\
		&= \frac{n!\,\pi_{k-1}}{(r_{k+1}-r_{k-1}-1)!}\,I_{u_{k-1},u_{k+1}}\big(F(y_k);p_k,q_k\big)\\
		&=\frac{n!\,\pi_{k-1}}{(r_{k+1}-r_{k-1}-1)!}\\
		&\times\sum_{J_k=0}^{q_k-1}\binom{p_k+q_k-1}{J_k}\big(F(y_k)-u_{k-1}\big)^{p_k+q_k-1-J_k}\big(u_{k+1}-F(y_k)\big)^{J_k}\\
		&=\frac{n!\,\pi_{k-1}}{(r_{k+1}-r_{k-1}-1)!}\\
		&\times\sum_{J_k=0}^{r_{k+1}-r_k-1}\binom{r_{k+1}-r_{k-1}-1}{J_k}\big(F(y_k)-u_{k-1}\big)^{r_{k+1}-r_{k-1}-1-J_k}\big(u_{k+1}-F(y_k)\big)^{J_k}\\
		&=n!\,\pi_{k-2}\underbrace{\sum_{J_k=0}^{r_{k+1}-r_k-1}\frac{\big(u_{k+1}-F(y_k)\big)^{J_k}}{J_k!}}_{\Sigma_{J_k}}\\
		&\times\int_{u_{k-2}}^{F(y_{k-1})}\frac{(u_{k-1}-u_{k-2})^{r_{k-1}-r_{k-2}-1}}{(r_{k-1}-r_{k-2}-1)!}\cdot\frac{(F(y_k)-u_{k-1})^{r_{k+1}-r_{k-1}-1-J_k}}{(r_{k+1}-r_{k-1}-1-J_k)!}\, du_{k-1}\\
		&=n!\,\pi_{k-2}\Sigma_{J_k}\\
		&\times\int_{u_{k-2}}^{F(y_{k-1})}\frac{(u_{k-1}-u_{k-2})^{r_{k-1}-r_{k-2}-1}\cdot(F(y_k)-u_{k-1})^{r_{k+1}-r_{k-1}-1-J_k}}{B(r_{k-1}-r_{k-2},r_{k+1}-r_{k-1}-J_k)\Gamma(r_{k+1}-r_{k-2}-J_k)}\, du_{k-1}.
	\end{align*}
	Again, by taking $p_{k-1}=r_{k-1}-r_{k-2},\, q_{k-1}=r_{k+1}-r_{k-1}-J_k$, then from Lemma \ref{lmm: Ibeta}
	\begin{align*}
		&=\frac{n!\,\pi_{k-2}\Sigma_{J_k}}{(r_{k+1}-r_{k-2}-J_k-1)!}\, I_{u_{k-2},F(y_k)}\Big(F(y_{k-1});p_{k-1},q_{k-1}\Big)\\
		&=\frac{n!\,\pi_{k-2}\Sigma_{J_k}}{(r_{k+1}-r_{k-2}-J_k-1)!}
		\sum_{J_{k-1}=0}^{q_{k-1}-1}\binom{p_{k-1}+q_{k-1}-1}{J_{k-1}}\\
		&\times\big(F(y_{k-1})-u_{k-2}\big)^{p_{k-1}+q_{k-1}-1-J_{k-1}}\big(u_k-F(y_{k-1})\big)^{J_{k-1}}\\
		&=\frac{n!\,\pi_{k-2}\Sigma_{J_k}}{(r_{k+1}-r_{k-2}-J_k-1)!}
		\sum_{J_{k-1}=0}^{r_{k+1}-r_{k-1}-J_k-1}\binom{r_{k+1}-r_{k-2}-J_k-1}{J_{k-1}}\\
		&\times\big(F(y_{k-1})-u_{k-2}\big)^{r_{k+1}-r_{k-2}-J_k-J_{k-1}-1}\big(u_k-F(y_{k-1})\big)^{J_{k-1}}\\
		&=n!\,\pi_{k-3}\Sigma_{J_k}\underbrace{\sum_{J_{k-1}=0}^{r_{k+1}-r_{k-1}-J_k-1}\frac{\big(F(y_k)-F(y_{k-1})\big)^{J_{k-1}}}{J_{k-1}!}}_{\Sigma_{J_{k-1}}}\\
		&\hspace{2cm}\times\int_{u_{k-3}}^{F(y_{k-2})}\frac{(u_{k-2}-u_{k-3})^{r_{k-2}-r_{k-3}-1}}{(r_{k-2}-r_{k-3}-1)!}\\
		&\hspace{2cm}\cdot\frac{(F(y_{k-1})-u_{k-2})^{r_{k+1}-r_{k-2}-J_k-J_{k-1}-1}}{(r_{k+1}-r_{k-2}-J_k-J_{k-1}-1)!}\, du_{k-2}\\
		&=n!\,\pi_{k-3}\Sigma_{J_k}\Sigma_{J_{k-1}}\int_{u_{k-3}}^{F(y_{k-2})}\frac{(u_{k-2}-u_{k-3})^{r_{k-2}-r_{k-3}-1}}{\Gamma(r_{k+1}-r_{k-3}-J_k-J_{k-1})}\\
		&\hspace{3cm}\times\frac{(F(y_{k-1})-u_{k-2})^{r_{k+1}-r_{k-2}-J_k-J_{k-1}-1}}{B(r_{k-2}-r_{k-3},r_{k+1}-r_{k-2}-J_k-J_{k-1})}\, du_{k-2},
	\end{align*}
	To identify the limits of the summations, we proceed one step further, and again by taking $p_{k-2}=r_{k-2}-r_{k-3},\, q_{k-2}=r_{k+1}-r_{k-2}-J_k-J_{k-1}$, from Lemma \ref{lmm: Ibeta}, we have
	\begin{align*}
		&=\frac{n!\,\pi_{k-3}\Sigma_{J_k}\Sigma_{J_{k-1}}}{(r_{k+1}-r_{k-3}-J_k-J_{k+1}-1)!}\, I_{u_{k-3},F(y_{k-1})}\Big(F(y_{k-2});p_{k-2},q_{k-2}\Big)\\
		&=\frac{n!\,\pi_{k-3}\Sigma_{J_k}\Sigma_{J_{k-1}}}{(r_{k+1}-r_{k-3}-J_k-J_{k-1}-1)!}
		\sum_{J_{k-2}=0}^{q_{k-2}-1}\binom{p_{k-2}+q_{k-2}-1}{J_{k-2}}\\
		&\times\big(F(y_{k-2})-u_{k-3}\big)^{p_{k-2}+q_{k-2}-1-J_{k-2}}\big(u_k-F(y_{k-1})\big)^{J_{k-2}}\\
		&=\frac{n!\,\pi_{k-3}\Sigma_{J_k}\Sigma_{J_{k-1}}}{(r_{k+1}-r_{k-3}-J_k-J_{k-1}-1)!}\\
		&\times\sum_{J_{k-2}=0}^{r_{k+1}-r_{k-2}-J_k-J_{k-1}-1}\binom{r_{k+1}-r_{k-3}-J_k-J_{k-1}-1}{J_{k-2}}\\
		&\times\big(F(y_{k-2})-u_{k-3}\big)^{r_{k+1}-r_{k-3}-J_k-J_{k-1}-J_{k-2}-1}\big(F(y_{k-1}-F(y_{k-1})\big)^{J_{k-2}}\\
		&=n!\,\pi_{k-4}\Sigma_{J_k}\Sigma_{J_{k-1}}\underbrace{\sum_{J_{k-2}=0}^{r_{k+1}-r_{k-2}-J_k-J_{k-1}-1}\frac{\big(F(y_{k-1})-F(y_{k-2})\big)^{J_{k-2}}}{J_{k-2}!}}_{\Sigma_{J_{k-2}}}\\
		&\hspace{3cm}\times\int_{u_{k-4}}^{F(y_{k-3})}\frac{(u_{k-3}-u_{k-4})^{r_{k-3}-r_{k-4}-1}}{(r_{k-3}-r_{k-4}-1)!}\\
		&\hspace{3cm}\cdot\frac{(F(y_{k-2})-u_{k-3})^{r_{k+1}-r_{k-3}-J_k-J_{k-1}-J_{k-2}-1}}{(r_{k+1}-r_{k-3}-J_k-J_{k-1}-J_{k-2}-1)!}\, du_{k-3}.
	\end{align*}
	Continuing this calculation recursively, we obtain
	\begin{align*}
		&=n!\Sigma_{J_k}\cdots\Sigma_{J_2}\int_{u_0}^{F(y_1)}\frac{(u_1-u_0)^{r_1-r_0-1}}{(r_1-r_0-1)!}\cdot\frac{(F(y_2)-u_1)^{r_{k+1}-r_1-1-\sum_{i=2}^kJ_i}}{(r_{k+1}-r_1-1-\sum_{i=2}^kJ_i)!}\,du_1\\
		&=n!\Sigma_{J_k}\cdots\Sigma_{J_2}\int_{0}^{F(y_1)}\frac{u_1^{r_1-1}(F(y_2)-u_1)^{r_{k+1}-r_1-1-\sum_{i=2}^kJ_i}}{B(r_1,r_{k+1}-r_1-\sum_{j=2}^kJ_i)(r_{k+1}-1-\sum_{i=2}^kJ_i)!}\,du_1.
	\end{align*}
	By taking $p_1=r_1, q_1=r_{k+1}-r_1-\sum_{i=2}^{k}J_i$, we have
	\begin{align*}
		&=n!\Sigma_{J_k}\cdots\Sigma_{J_2}\frac{I_{0,F(y_2)}\big(F(y-1);p_1,q_1\big)}{(r_{k+1}-1-\sum_{i=2}^{k}J_i)!}\\
		&=n!\Sigma_{J_k}\cdots\Sigma_{J_2}\frac{1}{(r_{k+1}-1-\sum_{i=2}^{k}J_i)!}\\
		&\times\sum_{J_1=0}^{q_1-1}\binom{p_1+q_1-1}{J_1}\Big(F(y_1)\Big)^{p_1+q_1-1-J_1}\Big(F(y_2)-F(y_1)\Big)^{J_1}\\
		&=n!\Sigma_{J_k}\cdots\Sigma_{J_2}\frac{1}{(r_{k+1}-1-\sum_{i=2}^{k}J_i)!}\sum_{J_1=0}^{r_{k+1}-r_1-1-\sum_{i=2}^{k}J_i}\binom{r_{k+1}-1-\sum_{i=2}^{k}J_i}{J_1}\\
		&\times\Big(F(y_1)\Big)^{r_{k+1}-1-\sum_{i=2}^{k}J_i}\Big(F(y_2)-F(y_1)\Big)^{J_1}\\
		&=n!\Sigma_{J_k}\cdots\Sigma_{J_2}\sum_{J_1=0}^{r_{k+1}-r_1-1-\sum_{i=2}^{k}J_i}\frac{\big(F(y_2)-F(y_1)\big)^{J_1}}{J_1!}\cdot\frac{\big(F(y_1)\big)^{r_{k+1}-1-\sum_{i=1}^{k}J_i}}{(r_{k+1}-1-\sum_{i=1}^{k}J_i)!}\\
		&=\sum_{J_k=0}^{r_{k+1}-r_k-1}\quad\sum_{J_{k-1}=0}^{r_{k+2}-r_{k-1}-J_k-1}\quad\sum_{J_{k-2}=0}^{r_{k+1}-r_{k-2}-J_k-J_{k-1}-1}\cdots\sum_{J_1=0}^{r_{k+1}-r_1-1-\sum_{i=2}^{k}J_i}\\
		&\frac{n!}{(r_{k+1}-1-\sum_{i=1}^{k}J_i)!J_1!\cdots J_k!}\\
		&\times\big(F(y_{k+1})-F(y_k)\big)^{J_k}\big(F(y_k)-F(y_{k-1})\big)^{J_{k-1}}\cdots\big(F(y_1)-F(y_0)\big)^{r_{k+1}-1-\sum_{i=1}^{k}J_i}
	\end{align*}
	where $F(y_0)=0$ and $F(y_{k+1})=1$. Then, taking $J_0=r_{k+1}-1-\sum_{i=1}^{k}J_i=n-\sum_{i=1}^{k}J_i$, we have
	\begin{align}\label{eq: CDF_y}
	&=\sum_{J_k=0}^{n-r_k}\quad\sum_{J_{k-1}=0}^{n-r_{k-1}-J_k}\quad\sum_{J_{k-2}=0}^{n-r_{k-2}-J_k-J_{k-1}}\cdots\sum_{J_1=0}^{n-r_1-\sum_{i=2}^kJ_i}\\
	&\quad\binom{n}{J_0,J_1,\ldots,J_k}\prod_{i=0}^{k}\Big(F\left(y_{i+1}\right)-F\left(y_{i}\right)\Big)^{J_i},\notag\\
	&s.t.\quad\sum_{i=0}^k J_i = n.\notag
    \end{align}
    Now, for arbitrary $y_1,\ldots,y_k$, we have 
    \begin{align}
    	&F_{(r_1,\ldots,r_k)}^n(y_1,\ldots,y_k)\notag\\
    	&=\P\left[X_{(1)}^n\le y_1,\ldots,X_{(k)}^n\le y_k\right]\notag\\
    	&=\P\left[X_{(1)}^n\le y_{(1)}^k,\ldots,X_{(k)}^n\le y_{(k)}^k\right]\notag\\
    	&\label{eq: CDF_range}=F_{(r_1,\ldots,r_k)}^n\left(y_{(1)}^k,\ldots,y_{(k)}^k\right),
    \end{align}
    and as $y_{(1)}^k\le\cdots\le y_{(k)}^k$, applying \eqref{eq: CDF_y} to \eqref{eq: CDF_range} returns
    \begin{align*}
    	&F_{(r_1,\ldots,r_k)}^n(y_1,\ldots,y_k)\\ 
    	&=\sum_{J_k=0}^{n-r_k}\quad\sum_{J_{k-1}=0}^{n-r_{k-1}-J_k}\quad\sum_{J_{k-2}=0}^{n-r_{k-2}-J_k-J_{k-1}}\cdots\sum_{J_1=0}^{n-r_1-\sum_{i=2}^kJ_i}\\
    	&\quad\binom{n}{J_0,J_1,\ldots,J_k}\prod_{i=0}^{k}\Big(F\left(y_{(i+1)}^k\right)-F\left(y_{(i)}^k\right)\Big)^{J_i},\\
    	&s.t.\quad\sum_{i=0}^k J_i = n.
    \end{align*}
\end{proof}
\begin{proof}[\bf Proof 2: Combinatorics]
First, we note
\begin{equation*}
	X_{(r)}^n\le y\iff r\le\sum_{i=1}^n\1_{(-\infty,y]}(X_i),
\end{equation*}
and so,
\begin{align}
	&F_{(r_1,\ldots,r_k)}^n(y_1,\ldots,y_k)\notag\\
	&=\P\left[X_{(1)}^n\le y_1,\ldots,X_{(k)}^n\le y_k\right]\notag\\
	&=\P\left[r_1\le\sum_{i=1}^n\1_{(-\infty,y_1]}(X_i)\,,\,\cdots\,,\, r_k\le\sum_{i=1}^n\1_{(-\infty,y_k]}(X_i)\right].\label{eq: P1}
\end{align}
Here, we consider the following intervals
\begin{center}
	$\; U_0\hspace{1.5cm} 
	U_1\hspace{1.4cm}
	\cdots\hspace{1.5cm}
	U_{k-1}\hspace{1.6cm}
	U_k$\\
	$-\infty=y_0$]---------[
	$y_1$]---------[$y_2
	\quad\cdots\quad 
	y_{k-1}$]---------[
	$y_k$]---------[$y_{k+1}=\infty$,
\end{center}
and denote
\begin{equation*}
	\#_k\;:=\;\#\{ i|X_i\in U_k\}\,=\,\sum_{i=1}^n\1_{]y_k,y_{k+1}[}(X_i),
\end{equation*}
where $\#$ denotes the cardinality of the set. Then, we have
\[r_k\le\sum_{i=1}^n\1_{(-\infty,y_k]}(X_i)\iff 0\le \#_k\le n-r_k.\]
Thus, equation \eqref{eq: P1} can be rewritten as follows.
\begin{align}		  
	&=\P\left[r_1\le\sum_{i=1}^n\1_{(-\infty,y_1]}(X_i)\,,\,\cdots\,,\, r_{k-1}\le\sum_{i=1}^n\1_{(-\infty,y_{k-1}]}(X_i),\, 0\le \#_k\le n-r_k\right]\notag\\
	&=\sum_{J_k=0}^{n-r_k}\P\left[r_1\le\sum_{i=1}^n\1_{(-\infty,y_1]}(X_i)\,,\,\cdots\,,\, r_{k-1}\le\sum_{i=1}^n\1_{(-\infty,y_{k-1}]}(X_i)\Big|\#_k=J_k\right]\notag\\
	&\times\P[\#_k=J_k]\notag\\	
	&=\sum_{J_k=0}^{n-r_k}\P\left[r_1\le\sum_{i=1}^n\1_{(-\infty,y_1]}(X_i)\,,\,\cdots\,,\, r_{k-1}\le\sum_{i=1}^n\1_{(-\infty,y_{k-1}]}(X_i)\Big|\#_k=J_k\right]\notag\\
	&\times\binom{n}{J_k}\Big(F\left(y_{k+1}\right)-F\left(y_k\right)\Big)^{J_k}.\label{eq: P2}
\end{align}
By denoting
\begin{equation*}
	\#_{k-1}:=\#\{ i|X_i\in U_{k-1}\}=\sum_{i=1}^n\1_{]y_{k-1},y_k[}(X_i),
\end{equation*}
we can continue \eqref{eq: P2} as follows.
\begin{align*}		  
	&=\sum_{J_k=0}^{n-r_k}\quad\sum_{J_{k-1}=0}^{n-J_k-r_{k-1}}\\
	&\P\left[r_1\le\sum_{i=1}^n\1_{(-\infty,y_1]}(X_i)\,,\,\cdots\,,\, r_{k-2}\le\sum_{i=1}^n\1_{(-\infty,y_{k-2}]}(X_i)\Big|\#_{k-1}=J_{k-1},\#_k=J_k\right]\notag\\
	&\times\P[\#_{k-1}=J_{k-1}\,|\,\#_k=J_k]\cdot\P[\#_k=J_k]\notag\\	
	&=\sum_{J_k=0}^{n-r_k}\sum_{J_{k-1}=0}^{n-J_k-r_{k-1}}\\
	&\P\left[r_1\le\sum_{i=1}^n\1_{(-\infty,y_1]}(X_i)\,,\,\cdots\,,\, r_{k-2}\le\sum_{i=1}^n\1_{(-\infty,y_{k-2}]}(X_i)\Big|\#_{k-1}=J_{k-1},\#_k=J_k\right]\notag\\
    &\times\binom{n}{J_k}\binom{n-J_k}{J_{k-1}}\Big(F\left(y_{k+1}\right)-F\left(y_k\right)\Big)^{J_k}\Big(F\left(y_k\right)-F\left(y_{k-1}\right)\Big)^{J_{k-1}}.
\end{align*}
By continuing this process, we obtain
\begin{align*}
	=&\sum_{J_k=0}^{n-r_k}\quad\sum_{J_{k-1}=0}^{n-r_{k-1}-J_k}\quad\sum_{J_{k-2}=0}^{n-r_{k-2}-J_k-J_{k-1}}\cdots\sum_{J_1=0}^{n-r_1-\sum_{i=2}^kJ_i}\\
	&\P[\#_1=J_1\,|\,\#_2=J_2,\ldots,\#_k=J_k]\\
	&\P[\#_2=J_2\,|\,\#_3=J_3,\ldots,\#_k=J_k]\\
	&\cdots\\
	&\P[\#_{k-1}=J_{k-1}\,|\,\#_k=J_k]\\
	&\P[\#_k=J_k]\\
	=&\sum_{J_k=0}^{n-r_k}\quad\sum_{J_{k-1}=0}^{n-r_{k-1}-J_k}\quad\sum_{J_{k-2}=0}^{n-r_{k-2}-J_k-J_{k-1}}\cdots\sum_{J_1=0}^{n-r_1-\sum_{i=2}^kJ_i}\\
	&\binom{n}{J_k}\binom{n-J_k}{J_{k-1}}\cdots\binom{n-\sum_{i=2}^{k}J_i}{J_1}\cdot\Big(F\left(y_1\right)-F\left(y_0\right)\Big)^{n-\sum_{i=1}^{k}J_i}\\
	&\times\Big(F\left(y_2\right)-F\left(y_1\right)\Big)^{J_1}\cdots\Big(F\left(y_{k+1}\right)-F\left(y_k\right)\Big)^{J_k},
\end{align*}
and by taking $J_0=n-\sum_{i=1}^{k}J_i$, this returns \eqref{eq: P0}.
\end{proof}
\begin{cor}
	Given the assumptions and notations in Theorem \ref{thm: CDF_Ordered}, if $x_1\le\cdots\le x_k$, then
	\begin{align}
		&F_{(r_1,\ldots,r_k)}^n(x_1,\ldots,x_k)\\ 
		&=\sum_{J_k=0}^{n-r_k}\quad\sum_{J_{k-1}=0}^{n-r_{k-1}-J_k}\quad\sum_{J_{k-2}=0}^{n-r_{k-2}-J_k-J_{k-1}}\cdots\sum_{J_1=0}^{n-r_1-\sum_{i=2}^kJ_i}\notag\\
		&\quad\binom{n}{J_0,J_1,\ldots,J_k}\prod_{i=0}^{k}\big(F\left(x_{i+1}\right)-F\left(x_i\right)\big)^{J_i},\notag\\
		&s.t.\quad\sum_{i=0}^k J_i = n.\notag
	\end{align}
\end{cor}
\section{Exact Distribution of Tail-Based Statistics}\label{sec:exact.dist}
Here, we apply the Corollary \ref{cor: fnn} to investigate the exact cumulative distribution of $\Lambda_n(p)$. We use a.s. to denote almost sure convergence.
\begin{pro}\label{prop: F_lambda}
Let $p\in(0,1)$ and $0<|\lambda|<1$. If the random variables $X_1,X_2,\ldots,X_n$ are i.i.d with common absolutely continuous distribution $F$, and the almost everywhere positive density function $f$, then for some $N_0\ge 1$, the cumulative distribution function of $\Lambda_n(p), n\ge N_0$ is
%\ \\
\begin{itemize}
\item[(i)] If $\lambda\E[X]>0$, then
\begin{align}\label{eq: FL_1}
&F_{\Lambda_n(p)}(\lambda) = 1 - n!\int_0^1\int_0^{u_n}\cdots\int_0^{u_3}\notag\\
&F\left[\left(\frac{1-\lambda}{\lambda}\right)\sum_{i=\ceil{np}}^nF^{-1}(u_i)-\sum_{i=2}^{\ceil{np}-1}F^{-1}(u_i)\right]\,du_2\cdots du_n.
\end{align}
%\ \\
\item[(ii)] If $\lambda\E[X]<0$, then
\begin{align}\label{eq: FL_2}
	&F_{\Lambda_n(p)}(\lambda) = n!\int_0^1\int_0^{u_n}\cdots\int_0^{u_3}\notag\\
	&F\left[\left(\frac{1-\lambda}{\lambda}\right)\sum_{i=\ceil{np}}^nF^{-1}(u_i)-\sum_{i=2}^{\ceil{np}-1}F^{-1}(u_i)\right]\,du_2\cdots du_n.
\end{align}
\end{itemize}
\end{pro}
\begin{proof}
For (i), let $0<\lambda, \E[X]<1$. Then by the strong low of larg numbers (LLN) $\sum_{i=1}^n X_i/n\overset{a.s.}{\longrightarrow}\E[X]$ and so, there are some $N_0\ge 1$ that for all $n\ge N_0$ we have $\sum_{i=1}^n X_i>0$. Next we have
\begin{gather}
\Lambda_n(p)\le\lambda\notag\\
\iff\frac{\sum_{i=\ceil{np}}^n X_{(i)}^n}{\sum_{i=1}^n X_i}\le\lambda\notag\\
\iff\sum_{i=\ceil{np}}^n X_{(i)}^n\le\lambda\sum_{i=1}^n X_i\notag\\
\iff\lambda\sum_{i=1}^{\ceil{np}-1} X_{(i)}^n\, -\, (1-\lambda)\sum_{i=\ceil{np}}^n X_{(i)}^n\, \ge\, 0\notag\\
\iff X_{(1)}^n\ge \left(\frac{1-\lambda}{\lambda}\right)\sum_{i=\ceil{np}}^n X_{(i)}^n-\sum_{i=2}^{\ceil{np}-1} X_{(i)}^n.\label{eq: ineq_lambda}
\end{gather}
Now, by denoting
$$
D_n(\lambda,p):=\left\{\bm{x}=(x_1,\ldots,x_n)\,\Bigg|\,
\begin{matrix}
\left(\frac{1-\lambda}{\lambda}\right)\sum_{i=\ceil{np}}^n x_i-\sum_{i=2}^{\ceil{np}-1} x_i\le x_1,\\
x_1<x_2<\cdots<x_n	
\end{matrix}
\right\}
\subset\R^n,
$$
from Corollary \ref{cor: fnn}, one can write
\begin{align*}
F_{\Lambda_n(p)}(\lambda)
&=\P[\Lambda_n(p)\le\lambda]\\
&=\P\left[\lambda\sum_{i=1}^{\ceil{np}-1} X_{(i)}^n\, -\, (1-\lambda)\sum_{i=\ceil{np}}^n X_{(i)}^n\, \ge\, 0\right]\\
&=\int_{D_n(\lambda,p)} f^n_{(1,\ldots,n)}(\bm{x})\, d\bm{x}\\
&=n!\underset{D_n(\lambda,p)}{\idotsint} \left(\prod_{i=1}^n f(x_i)\right)\,dx_1\cdots dx_n\\
&=n!\int_{-\infty}^{\infty}
\int_{-\infty}^{x_n}
\cdots
\int_{-\infty}^{x_3}
\int_{\left(\frac{1-\lambda}{\lambda}\right)\sum_{i=\ceil{np}}^n x_i-\sum_{i=2}^{\ceil{np}-1} x_i}^{x_2}\\
&\qquad\left(\prod_{i=1}^n f(x_i)\right)\,dx_1\cdots dx_n.
\end{align*}
Then, since $F$ is almost everywhere differentiable and invertible, by changing the variables $u_i=F(x_i)$ or $x_i=F^{-1}(u_i)$, for every $i=1,\ldots,n$ we have
\begin{align*}
&=n!\int_0^1
\int_0^{u_n}
\cdots
\int_0^{u_3}
\int_{F\left[\left(\frac{1-\lambda}{\lambda}\right)\sum_{i=\ceil{np}}^n F^{-1}(u_i)-\sum_{i=2}^{\ceil{np}-1} F^{-1}(u_i)\right]}^{u_2}\,du_1\cdots du_n\\
&=n!\int_0^1
\int_0^{u_n}
\cdots
\int_0^{u_3}
\int_0^{u_2}\,du_1\cdots du_n\\
&-n!\int_0^1
\int_0^{u_n}
\cdots
\int_0^{u_3}
\int_0^{F\left[\left(\frac{1-\lambda}{\lambda}\right)\sum_{i=\ceil{np}}^n F^{-1}(u_i)-\sum_{i=2}^{\ceil{np}-1} F^{-1}(u_i)\right]}\,du_1\cdots du_n,
\end{align*}
and this yields \eqref{eq: FL_1}. If $-1<\lambda, \E[X]<0$, then similarly \eqref{eq: ineq_lambda} is valid and so we have the same result.

For (ii), let $\lambda>0, \E[X]<0$. Similar to the proof of (i), there are some $N_0\ge 1$ that for all $n\ge N_0$ we have $\sum_{i=1}^n X_i<0$. Next, for this case, one can see
\begin{equation}\label{eq: ineq_lambda2}
\Lambda_n(p)\ge\lambda
\iff
X_{(1)}^n\ge \left(\frac{1-\lambda}{\lambda}\right)\sum_{i=\ceil{np}}^n X_{(i)}^n-\sum_{i=2}^{\ceil{np}-1} X_{(i)}^n,
\end{equation}
and so, in this case we have
\begin{align*}
F_{\Lambda_n(p)}(\lambda)
&=\P[\Lambda_n(p)\le\lambda]\\
&=1-\P[\Lambda_n(p)\ge\lambda]\\
&=1-\P\left[X_{(1)}^n\ge \left(\frac{1-\lambda}{\lambda}\right)\sum_{i=\ceil{np}}^n X_{(i)}^n-\sum_{i=2}^{\ceil{np}-1} X_{(i)}^n\right]\\
&=1-\int_{D_n(\lambda,p)} f^n_{(1,\ldots,n)}(\bm{x})\, d\bm{x}.
\end{align*}
Now, proceeding with similar calculation of part (i), from this final integral we have \eqref{eq: FL_2}. If $\lambda<0, \E[X]>0$, then similarly \eqref{eq: ineq_lambda2} is valid, and so, we have the same result.
\end{proof}
\section{Convergence of Tail-Based Statistics}\label{sec:convergence}
As the explicit form of the exact distribution functions \eqref{eq: FL_1} and \eqref{eq: FL_2} include multiple integrals, they are not computationally suitable for larg nambers of $n$. So, we need to investigate further for the asymptotic behavior and distribution here.
\begin{lmm}\label{lmm: topfrac}
	Let $X_1,X_2,\ldots,X_n$ be i.i.d random variables with common absolutely continuous distribution $F$. Then, for all $p\in(0,1)$ that $F$ is continuous at its $p$th quantile $q_p$, we have
	\begin{equation}
		\frac{1}{n}\sum_{i=1}^{n}X_i\1_{\{ X_i\ge Q_n(p)\}}
		\overset{a.s.}{\longrightarrow}
		\E\left[X_1\1_{\{ X_1\ge q_p\}}\right],
	\end{equation}
	where $Q_n(p)$ is the $p$th quantile of $\{ X_i\}_{i=1}^n$.
\end{lmm}
\begin{proof}
	Let
	\begin{align*}
		U_n(p) &= \frac{1}{n}\sum_{i=1}^{n}X_i\1_{\{ X_i\ge Q_n(p)\}},\\
		V_n(p) &= \frac{1}{n}\sum_{i=1}^{n}X_i\1_{\{ X_i\ge q_p\}}.
	\end{align*}
	By the strong LLN we have
	\begin{equation}\label{eq: Vn_Converge}
		V_n(p)
		\overset{a.s.}{\longrightarrow}
		\E\left[X_1\1_{\{ X_1\ge q_p\}}\right].
	\end{equation}
	On the other hand, it is shown in \cite{fabian1985introduction} that, given the continuity of $F$ on $q_p$, we have
	\begin{equation}\label{eq: Qnq}
		Q_n(p)\overset{a.s.}{\longrightarrow}q_p.
	\end{equation}
	Next, we can write
	\begin{align*}
		|U_n(p)-V_n(p)| 
		&\le \frac{1}{n}\sum_{i=1}^{n}|X_i|\cdot|\1_{\{ X_i\ge Q_n(p)\}} - \1_{\{ X_i\ge q_p\}}|\\
		&= \frac{1}{n}\sum_{i=1}^{n}|X_i|\,\1_{\{ Q_n(p)\wedge q_p\le X_i< Q_n(p)\vee q_p\}}\\
		&= \frac{1}{n}\sum_{i=1}^{n}|X_i|\,\1_{\{ a_n(p)\le X_i< b_n(p)\}},
	\end{align*}
	where $\wedge$ and $\vee$ are respectively minimum and maximum. Here, from \eqref{eq: Qnq}, for both $a_n(p)=Q_n(p)\wedge q_p$ and $b_n(p)=Q_n(p)\vee q_p$ we have
	\begin{equation*}
		a_n(p), b_n(p)\overset{a.s.}{\longrightarrow}q_p.
	\end{equation*}
	So, almost surely, for all arbitrary $\varepsilon>0$, there are some $N_1^\varepsilon(p)>0$ that for all $n\ge N_1^\varepsilon(p)$, we have
	\begin{equation*}
		\1_{[a_n(p),b_n(p))}(x)=0,\qquad \forall x\in\R\setminus(q_p-\varepsilon,q_p+\varepsilon),
	\end{equation*}
	and so,
	\begin{equation*}
		\1_{[a_n(p),b_n(p))}(x)\le\1_{(q_p-\varepsilon\,,\,q_p+\varepsilon)}(x).
	\end{equation*}
	Hence, for $i\ge 1$
	\begin{equation*}
		\1_{[a_n(p),b_n(p))}(X_i)\le\1_{(q_p-\varepsilon\,,\,q_p+\varepsilon)}(X_i).
	\end{equation*}
	So, for $n\ge N_1^\varepsilon(p)$, we have
	\begin{align*}
		|U_n(p) - V_n(p)| 
		&\le \frac{1}{n}\sum_{i=1}^{n}|X_i|\,\1_{\{ [a_n(p),b_n(p))\}}(X_i)\\
		&\le \frac{1}{n}\sum_{i=1}^{n}|X_i|\,\1_{(q_p-\varepsilon\,,\,q_p+\varepsilon)}(X_i).
	\end{align*}
	Again, by the strong LLN
	\begin{equation*}
		\frac{1}{n}\sum_{i=1}^{n}|X_i|\,\1_{(q_p-\varepsilon\,,\,q_p+\varepsilon)}(X_i)
		\overset{a.s.}{\longrightarrow}
		\E\left[|X_1|\, \1_{(q_p-\varepsilon\,,\,q_p+\varepsilon)}(X_1)\right].
	\end{equation*}
	That is, almost surely, for all $\varepsilon>0$, there are some $N_2^\varepsilon(p)>0$ that for all $n\ge N_2^\varepsilon(p)$, we have
	\begin{align*}
		|U_n(p) - V_n(p)| &\le \E\left[|X_1|\, \1_{(q_p-\varepsilon\,,\,q_p+\varepsilon)}(X_1)\right]\\
		&=\int_{-\varepsilon}^{\varepsilon}|q_p-x|\cdot f(q_p-x)\,dx\\
		&\le 2\varepsilon\max_{(-\varepsilon,\varepsilon)} |q_p-x|\cdot f(q_p-x)\\
		&\le 2\varepsilon M,
	\end{align*}
	where $f$ is the Radon-Nikodym derivative of $F$, i.e., the probability density function of $X_1$. We note, as $f$ is integrable in an interval $(q_p-\ell,q_p+\ell)$ around $x=q_p$, we have $M=\max_{(-\ell,\ell)} |q_p-x|\cdot f(q_p-x)<\infty$. So,
	\begin{equation}\label{eq: Cauchy-as}
		|U_n(p) - V_n(p)|
		\overset{a.s.}{\longrightarrow}
		0.
	\end{equation}
	Now, \eqref{eq: Vn_Converge} and \eqref{eq: Cauchy-as} prove the Theorem since the intersection of two events, each with probability 1, also has probability 1.
\end{proof}
\begin{cor}\label{cor: UVas}
	By the assumptions and notations of the Lemma \ref{lmm: topfrac} and its proof, for $n\to\infty$ we have   
	$
	U_n(p) = \frac{1}{n}\sum_{i=1}^{n}X_i\1_{\{ X_i\ge Q_n(p)\}}
	$ 
	is almost surely (a.s.) close to the process
	$
	V_n(p) = \frac{1}{n}\sum_{i=1}^{n}X_i\1_{\{ X_i\ge q_p\}}
	$.
\end{cor}
Next, we note
\begin{equation}\label{eq: LQn}
	\Lambda_n(p)=\frac{\frac{1}{n}\sum_{i=1}^{n}X_i\1_{\{ X_i\ge Q_n(p)\}}}{\frac{1}{n}\sum_{i=1}^{n}X_i},
\end{equation}
and so, we have the following theorem as a straightforward consequence of the Lemma \ref{lmm: topfrac} and the strong LLN result that 
$\frac{1}{n}\sum_{i=1}^{n}X_i
\overset{a.s.}{\longrightarrow}
\mu
$
.
\begin{thm}\label{thm: LambdaConverg}
	Let $X_1,X_2,\ldots,X_n$ be i.i.d random variables with common absolutely continuous distribution $F$, and $\mu=\E[X_1]\neq 0$. Then, for all $p\in(0,1)$ that $F$ is continuous at its $p$th quantile $q_p$, we have
	\begin{equation}\label{eq: asconvergence}
		\Lambda_n(p)
		\overset{a.s.}{\longrightarrow}
		\frac{a_{q_p}}{\mu},
	\end{equation}
	where $a_{q_p} = \E[X_1\,\1_{\{ X_1\ge q_p\}}]$.
\end{thm}
\section{Asymptotic Distribution of Numerator}\label{sect: Asymp.Dist.Num}
Considering the explicit form of $\Lambda_n(p)$ given by equation \eqref{eq: LQn}, it has a ratio distribution for large $n\to\infty$. If $\{X_i\}_{i\ge1}$ are i.i.d with $\E[X_i]=\mu,\Var[X_i]=\sigma^2$, Then, by the central limit theorem (CLT), the denominator of the fraction converges to a normally distributed random variable. That is, for $n\to\infty$
\begin{equation*}
Z_n=\frac{1}{n}\sum_{i=1}^{n}X_i
\sim
\mathcal{N}(\mu,\sigma^2/n).
\end{equation*}
On the other hand, as $\{X_i\,\1_{\{X_i\ge Q_n(p)\}}\}_{i\ge1}$ are not independent random variables, the CLT is not applicable for the asymptotic distribution of the numerator of the fraction 
$U_n(p)$, 
%$$
%U_n(p) = \frac{1}{n}\sum_{i=1}^{n}X_i\1_{\{ X_i\ge Q_n(p)\}}
%$$
even though it converges almost surely by the Lemma \ref{lmm: topfrac}. Considering the literature on ratio distributions and the multiple integral involved in the explicit form of the exact distribution functions, \eqref{eq: FL_1} and \eqref{eq: FL_2}, a close form of the $\Lambda_n(p)$ distribution is so complicated (case-dependent) to characterize in general for large $n\to\infty$.

To overcome these difficulties, the asymptotic distribution of $U_n(p)$ is required. To this, we apply the asymptotic distribution of $Q_n(p)$ and the law of total probability. The asymptotic normality of the distribution of $Q_n(p)$ for $n\to\infty$ was investigated by \cite{serfling2009approximation,bahadur1966note}
\begin{equation}\label{eq: Qn_asymp}
	Q_n(p)\sim\mathcal{N}\left(q_p\,,\,\frac{p(1-p)}{nf^2(q_p)}\right),\quad n\to\infty.
\end{equation}
So, applying this distribution, one can see
\[
f_{Q_n(p)}(q)\;=\; \frac{e^{-\frac{nf^2(q_p)}{2p(1-p)}(q-q_p)^2}}
{\sqrt{\frac{2\pi p(1-p)}{nf^2(q_p)}}}
\;\underset{n\to\infty}{\longrightarrow}\;
\delta_{q_p}(q).
\] 
While there are plenty studies for the ratio distributions of two Gaussian processes, the literatures for those ratios that numerator or denumerator are non-Gaussian are not that rich and also show sevear difficulties to have an explicit form of those ratio distribution. Then, a very straight forward question one may ask that is:
\begin{center}
	{\it ``Does $U_n$ have an asymptotic normality in distribution?''}
\end{center}
The following theorem reveals a positive response, and the fact behind it.
\begin{thm}\label{thm: asymp.norm}
	Let $X_1,X_2,\ldots,X_n$ be i.i.d square integrable random variables, i.e., $\E[X_1^2]<\infty$, with common distribution $F$ continuous at $q_p$. Then, for $n\to\infty$, the process $U_n$ admits the asymptotic normal distribution
	\begin{equation}\label{eq: UasympN}
     U_n(p)\,\sim\,\mathcal{N}\Big(a_{q_p}\,,\,\Big(
     (b^+_{q_p})^2 + (b^-_{q_p})^2 + 2a^+_{q_p}a^-_{q_p}
     \Big)\Big/ n\Big),
	\end{equation}
	where $a^+_{q_p},b^+_{q_p}$ are the expectation and standard deviation of $X_i^+\1_{\{X_i\geq q_p\}}$, and $a^-_{q_p},b^-_{q_p}$ are the expectation and standard deviation of $X_i^-\1_{\{X_i\geq q_p\}}$.
\end{thm}
\begin{proof}
First, considering \eqref{eq: Qn_asymp}, for $n\to\infty$ there are some $r_n>0$ that $r_n\to 0$ and
\[
\P[Q_n\in B_{r_n}(q_p)]\approx 1.
\]
So, for $n\to\infty$ almost surely 
$$
q_p - r_n\le Q_n\le q_p + r_n,
$$
and so,
\[
\1_{\{X_i\ge q_p + r_n\}}\le
\1_{\{X_i\ge Q_n\}}\le
\1_{\{X_i\ge q_p - r_n\}}.
\]
Next, we have $X_i = X_i^+ - X_i^-$ where $X_i^+ = \max\{X_i,0\}$ and $X_i^- = \max\{-X_i,0\}$, and also
\begin{align*}
X_i^+\1_{\{X_i\ge q_p + r_n\}}\le
X_i^+\1_{\{X_i\ge Q_n\}}\le
X_i^+\1_{\{X_i\ge q_p - r_n\}},\\
X_i^-\1_{\{X_i\ge q_p + r_n\}}\le
X_i^-\1_{\{X_i\ge Q_n\}}\le
X_i^-\1_{\{X_i\ge q_p - r_n\}}.
\end{align*}
Hence,
\begin{align}
&X_i^+\1_{\{X_i\ge q_p + r_n\}}
-
X_i^-\1_{\{X_i\ge q_p - r_n\}}\notag\\
&\le
X_i\1_{\{X_i\ge Q_n\}}=
(X_i^+ - X_i^-)\1_{\{X_i\ge Q_n\}}\notag\\
&\le
X_i^+\1_{\{X_i\ge q_p - r_n\}}
-
X_i^-\1_{\{X_i\ge q_p + r_n\}},\label{ineqXi}
\end{align}
and so,
\[
W^+_+(n) - W^-_-(n)\leq
U_n\leq
W^+_-(n) - W^-_+(n),
\]
where
\begin{align*}
W^+_+(n) &= \frac{1}{n}\sum_{i=1}^{n}X_i^+\1_{\{ X_i\ge q_p + r_n\}},\\
W^+_-(n) &= \frac{1}{n}\sum_{i=1}^{n}X_i^+\1_{\{ X_i\ge q_p - r_n\}},\\
W^-_+(n) &= \frac{1}{n}\sum_{i=1}^{n}X_i^-\1_{\{ X_i\ge q_p + r_n\}},\\
W^-_-(n) &= \frac{1}{n}\sum_{i=1}^{n}X_i^-\1_{\{ X_i\ge q_p - r_n\}},
\end{align*}
are all Gaussian processes. So, the processes $W_n := W^+_+(n) - W^-_-(n)$ and $V_n := W^+_-(n) - W^-_+(n)$ are also Gaussian and 
\begin{equation}\label{ineqU}
	W_n\leq U_n\leq V_n.
\end{equation}
Now,
\begin{align*}
\E[W_n] &= \E[X_1^+\1_{\{X_1\geq q_p + r_n\}}] - \E[X_1^-\1_{\{X_1\geq q_p - r_n\}}] =: a^+_+(n,q_p) - a^-_-(n,q_p),\\
\E[V_n] &= \E[X_1^+\1_{\{X_1\geq q_p - r_n\}}] - \E[X_1^-\1_{\{X_1\geq q_p + r_n\}}] =: a^+_-(n,q_p) - a^-_+(n,q_p),
\end{align*}
and
\begin{align*}
\Var[W_n] &= \frac{1}{n}
\Big(
\Var[X_1^+\1_{\{X_1\geq q_p + r_n\}}] 
+\Var[X_1^-\1_{\{X_1\geq q_p - r_n\}}]\\
&\qquad\quad -2\Cov[X_1^+\1_{\{X_1\geq q_p + r_n\}},X_1^-\1_{\{X_1\geq q_p - r_n\}}]
\Big)\\
&=\frac{1}{n}\Big(
(b^+_+)^2(n,q_p) + (b^-_-)^2(n,q_p) + 2a^+_+(n,q_p)a^-_-(n,q_p)
\Big).
\end{align*}
Similarly,
\begin{align*}
	\Var[V_n] &= \frac{1}{n}
	\Big(
	\Var[X_1^+\1_{\{X_1\geq q_p - r_n\}}] 
	+\Var[X_1^-\1_{\{X_1\geq q_p + r_n\}}]\\
	&\qquad\quad -2\Cov[X_1^+\1_{\{X_1\geq q_p - r_n\}},X_1^-\1_{\{X_1\geq q_p + r_n\}}]
	\Big)\\
	&=\frac{1}{n}\Big(
	(b^+_-)^2(n,q_p) + (b^-_+)^2(n,q_p) + 2a^+_-(n,q_p)a^-_+(n,q_p)
	\Big).
\end{align*}
Then, for $n\to\infty$
\begin{align*}
\E[W_n],\E[V_n]&\approx a_{q_p},\\
\Var[W_n],\Var[V_n]&\approx\frac{1}{n}\Big(
(b^+_{q_p})^2 + (b^-_{q_p})^2 + 2a^+_{q_p}a^-_{q_p}
\Big).
\end{align*}
These and \eqref{ineqU} proves this Theorem. 
\end{proof}
\section{Asymptotic Distribution of Tail-Based Statistic}\label{sect: Asymp.Dist.Lamb}
%
%We note that,  we can exert the random variable $V_n(p)$, given by Corollary \ref{cor: UVas}, which is almost surely close to it. Indeed, as $\{X_i\,\1_{\{X_i\ge q_p\}}\}_{i\ge1}$ are i.i.d, by the CLT, for $n\to\infty$ we have
%$$
%V_n(p) = \frac{1}{n}\sum_{i=1}^{n}X_i\1_{\{ X_i\ge q_p\}}
%\sim
%\mathcal{N}(a_{q_p},b_{q_p}^2/n),
%$$
%where $b_{q_p}^2 = \Var[X_1\1_{\{ X_1\ge q_p\}}]$. This, Theorem \ref{thm: LambdaConverg} and Corollary \ref{cor: UVas} reveal that for $\mu\neq 0$ and $n\to\infty$
%\begin{gather}
%\Upsilon_n(p)=\frac{V_n(p)}{Z_n}
%\sim
%\frac{\mathcal{N}(a_{q_p},b_{q_p}^2/n)}
%{\mathcal{N}(\mu,\sigma^2/n)},\label{eq: ratiodist0}\\
%%
%\Upsilon_n(p)
%\overset{a.s.}{\longrightarrow}
%\frac{a_{q_p}}{\mu},\\
%%
%|\Upsilon_n(p) - \Lambda_n(p)|
%\overset{a.s.}{\longrightarrow}
%0.\label{eq: ascauchy}
%\end{gather}
The previous section discussed how $U_n$ is asymptotically a normal process $\mathcal{N}(\mu_n(p),\sigma^2_n(p)/n)$, that for known large-size sample distributions we have  $\mu_n(p)\approx a_{q_p}$ and $\sigma^2_n(p)\approx (b^+_{q_p})^2 + (b^-_{q_p})^2 + 2a^+_{q_p}a^-_{q_p}
$. Thus, the process $\Lambda_n$ has a ratio distribution of two correlated, noncentral, normally distributed processes, $U_n$ and $Z_n$. Here, we aim to identify this ratio distribution.
The method used in the following lemma was initiated by \cite{hinkley1969ratio} and further developed by \cite{springer1979algebra}, who transformed the variables into a ratio of two uncorrelated normal processes with a constant offset. \cite{geary1930frequency} showed that these ratios can be “almost Gaussian” under certain restrictions, \cite{fieller1932distribution} provided an exact analysis, and \cite{pham2006density} examined them comprehensively. However, their computational combinations and associated complexities must be taken into account. \cite{hinkley1969ratio} also developed exact results for the correlated case. By transforming the variables to be uncorrelated, however, one can simply apply Hinkley’s formula rather than resorting to more complicated expressions.
\begin{lmm}\label{lmm: COV[u,z]}
	Let $X_1,X_2,\ldots,X_n$ be i.i.d square integrable random variables, i.e., $\E[X_1^2]<\infty$, with common absolutely continuous distribution $F$. Then, for $n\to\infty$,
\begin{align}
	C_n(p) &= \Cov[U_n,Z_n] = c_n(p)/n,\\
	\rho_n(p) &= \Cor[U_n,Z_n] = \frac{c_n(p)}{\sigma\cdot\sigma_n(p)},\\
	c_n(p)&\approx a^{(2)}_{q_p} - \mu a_{q_p} = b^2_{q_p}-\tilde a_{q_p}a_{q_p},
\end{align}
where $a^{(2)}_{q_p} = \E[X_i^2\1_{\{X_i\geq q_p\}}]$ and $\tilde a_{q_p} = \mu - a_{q_p} = \E[X_i\1_{\{X_i\leq q_p\}}]$.
\end{lmm}
\begin{proof}
\begin{align*}
	\Cov[U_n,Z_n] 
	&= \frac{1}{n^2}\Cov\left[
	\sum_{i=1}^{n}X_i\1_{\{X_i\geq Q_n\}},
	\sum_{j=1}^{n}X_j
	\right]\\
	&= \frac{1}{n^2}\sum_{i,j=1}^{n}\Cov[X_i\1_{\{X_i\geq Q_n\}},X_j].
\end{align*}
From \eqref{ineqXi} we have
\begin{align*}
	&X_j^+X_i^+\1_{\{X_i\ge q_p + r_n\}}
	-
	X_j^+X_i^-\1_{\{X_i\ge q_p - r_n\}}\\
	&\le
	X_j^+X_i\1_{\{X_i\ge Q_n\}}\\
	&\le
	X_j^+X_i^+\1_{\{X_i\ge q_p - r_n\}}
	-
	X_j^+X_i^-\1_{\{X_i\ge q_p + r_n\}},
\end{align*}
and also,
\begin{align*}
	&X_j^-X_i^+\1_{\{X_i\ge q_p + r_n\}}
	-
	X_j^-X_i^-\1_{\{X_i\ge q_p - r_n\}}\\
	&\le
	X_j^-X_i\1_{\{X_i\ge Q_n\}}\\
	&\le
	X_j^-X_i^+\1_{\{X_i\ge q_p - r_n\}}
	-
	X_j^-X_i^-\1_{\{X_i\ge q_p + r_n\}}.
\end{align*}
So,
\begin{align*}
	&X_j^+X_i^+\1_{\{X_i\ge q_p + r_n\}}
	-
	X_j^+X_i^-\1_{\{X_i\ge q_p - r_n\}}\\
	&-
	X_j^-X_i^+\1_{\{X_i\ge q_p - r_n\}}
	+
	X_j^-X_i^-\1_{\{X_i\ge q_p + r_n\}}\\
	&\le
	X_jX_i\1_{\{X_i\ge Q_n\}}
	=(X_j^+ - X_j^-)X_i\1_{\{X_i\ge Q_n\}}\\
	&\le
	X_j^+X_i^+\1_{\{X_i\ge q_p - r_n\}}
	-
	X_j^+X_i^-\1_{\{X_i\ge q_p + r_n\}}\\
	&-
	X_j^-X_i^+\1_{\{X_i\ge q_p + r_n\}}
	+
	X_j^-X_i^-\1_{\{X_i\ge q_p - r_n\}},
\end{align*}
and for $n\to\infty$ we have
\begin{align*}
	&\E[X_jX_i\1_{\{X_i\geq Q_n\}}]\approx\E[X_jX_i\1_{\{X_i\geq q_p\}}]\\
	&=
	\begin{cases}
		\E[X_i^2\1_{\{X_i\geq q_p\}}] = a^{(2)}_{q_p} & i=j\\
		\E[X_j]\E[X_i\1_{\{X_i\geq q_p\}}]=\mu a_{q_p} & i\leq j.\\
	\end{cases}
\end{align*}
Hence,
\begin{align*}
	\Cov[X_i\1_{\{X_i\geq Q_n\}}, X_j]
	\approx
	\begin{cases}
		a^{(2)}_{q_p} - \mu a_{q_p} & i=j\\
		0 & i\neq j,\\
	\end{cases}
\end{align*}
and so,
\[
\Cov[U_n,Z_n]=\left(a^{(2)}_{q_p} - \mu a_{q_p}\right)\Big/n.
\]
\end{proof}
%
%-----------------------------------\\
\begin{lmm}\label{lmm: offset}
	For $\mu\neq 0$, the process
	$$
	\widetilde{\Lambda}_n=\Lambda_n-\frac{c_n}{\sigma^2}
	$$
	is a ratio of uncorrelated noncentral Gaussian processes, and has the probability density function
	\begin{align}\label{eq: pdftild}
		f_{\widetilde{\Lambda}_n}(t)=
		&\;
		\sqrt{\frac{n}{2\pi\sigma^2\tilde \sigma^2_n}}
		\cdot\frac{\widetilde{B}(t)}{\widetilde{A}^3(t)}
		\cdot \exp\left[{-\frac{n}{2}\cdot\frac{\mu^2}{\sigma^2}\cdot\frac{\left(t-{\tilde\mu_n}/{\mu}\right)^2}{t^2+{\tilde\sigma^2_n}/{\sigma^2}}}\right]
		\erf\left(\frac{\widetilde{B}(t)}{\widetilde{A}(t)}\sqrt{\frac{n}{2}}\right)\notag\\
		&+ \frac{e^{-\frac{n}{2}r^2_n}}{{\pi\sigma\tilde\sigma_n}\widetilde{A}^2(t)},
	\end{align}
	where
	\begin{gather*}
		\tilde{\mu}_n = \mu_n - \mu c_n/\sigma^2,\quad
		\tilde{\sigma}_n = \sqrt{\sigma^2_n - c^2_n/\sigma^2},\quad
		r^2_n = \tilde{\mu}^2_n/\tilde{\sigma}^2_n + \mu^2/\sigma^2\\
		\widetilde{A}(t) =\sqrt{\frac{t^2}{\tilde{\sigma}^2_n} + \frac{1}{\sigma^2}},\quad
		\widetilde{B}(t) =\frac{\tilde{\mu}_n}{\tilde{\sigma}_n^2}\, t + \frac{\mu}{\sigma^2}.
	\end{gather*}
\end{lmm}
\begin{proof}
	We have
	\begin{equation}\label{eq: Tfrac1}
		\Lambda_n=
		\frac{\mu_n + u_n}
		{\mu + z_n},
	\end{equation}
	where $u_n(p)$ and $z_n$ are corellated central Gaussian variables with variances $\sigma_n^2/n$ and  $\sigma^2/n$ respectively, and correlation $\rho_n$. Now, we apply the Geary-Hinkley transformation. Let
	\begin{align*}
		\tilde{u}_n
		&=u_n-\rho_{u,z}\frac{\sigma_u}{\sigma_z}z_n\\
		&=u_n-\rho_n\frac{\sigma_n}{\sigma}z_n\\
		&=u_n-\frac{c_n}{\sigma^2}z_n,
	\end{align*}
	then $\tilde{u}_n$ and $z_n$ are uncorrelated, and by
	\begin{gather*}
		\E[\tilde{u}_n]=\mu_n - \mu\cdot\frac{c_n}{\sigma^2}=:\tilde\mu_n,\\	
		\Var[\tilde{u}_n]=
		(\sigma^2_n - c^2_n/\sigma^2)/n=:
		\tilde\sigma^2_n/n,
	\end{gather*}
	we have
	$$
	\Lambda_n=
	\frac{\tilde\mu_n + \tilde{u}_n}
	{\mu + z_n}
	+\frac{c_n}{\sigma^2}.
	$$
	Hence,
	$$
	\widetilde{\Lambda}_n
	=\Lambda_n-\frac{c_n}{\sigma^2}
	=\frac{\tilde\mu_n + \tilde{u}_n}
	{\mu + z_n},
	$$
	is a ratio of uncorrelated noncenteral Gaussian process, and so by \cite{hinkley1969ratio} has the probability density function
	\begin{equation}\label{eq: f_Yn}
		f_{\widetilde{\Lambda}_n}(t)\;
		= n\,\frac{\exp(-R^2/2)}{2\pi\sigma_n \sigma\,a^2(t)}\,
		\left[\sqrt{2\pi}\,\frac{b(t)}{a(t)}\,\exp\left(\frac{b^2(t)}{2a^2(t)}\right)\erf\left(\frac{b(t)}{\sqrt{2}a(t)}\right) + 2\,\right],
	\end{equation}
	where
	\begin{align*}
		R^2 &= n\left(\frac{\tilde{\mu}^2_n}{\tilde{\sigma}^2_n} + \frac{\mu^2}{\sigma^2}\right) = n r^2_n,\\
		a(t) &= \sqrt{n\left(\frac{t^2}{\tilde{\sigma}^2_n} + \frac{1}{\sigma^2}\right)} = \sqrt{n}\widetilde{A}(t),\\
		b(t) &= n\left(\frac{\tilde{\mu}_n}{\tilde{\sigma}^2_n}\, t + \frac{\mu}{\sigma^2}\right) = n\widetilde{B}(t).
	\end{align*}
	Finally, since
	\begin{equation*}
		\frac12\left(R^2 - \frac{b^2(t)}{a^2(t)}\right)
		=\frac{n}{2}\left(R^2 - \frac{B^2(t)}{A^2(t)}\right)
		=\frac{n}{2}\cdot\frac{\mu^2}{\sigma^2}\cdot\frac{\left(t-\tilde\mu_n/{\mu}\right)^2}{t^2+{\tilde\sigma^2_n}/{\sigma^2}},
	\end{equation*}
	\eqref{eq: f_Yn} results \eqref{eq: pdftild}. 
\end{proof}
Compensating the constant offset $-c_n/\sigma^2$ of the Lemma \ref{lmm: offset}, by changing variable $t\mapsto t - c_n/\sigma^2$ we have the following theorem.
\begin{thm}\label{thm: exactratio}
	For $\mu\neq 0$ and $n\to\infty$, the (asymptotic) probability density function of $\Lambda_n$ is
	\begin{align}
		f_{\Lambda_n}(t)=
		&\;\sqrt{\frac{n}{2\pi(\sigma^2\sigma^2_n-c_n^2)}}
		\cdot\frac{B(t)}{A^3(t)}
		\;\erf\left(\frac{B(t)}{A(t)}\sqrt{\frac{n}{2}}\right)\notag\\
		&\times \exp\left[{-\frac{n}{2}\cdot\frac{\mu^2}{\sigma^2}\cdot\frac{(t-\mu_n/\mu)^2}{(t-\sigma_n/\sigma)^2 + 2t(1-\rho_n)\sigma_n/\sigma}}\right]
		\notag\\
		&+ \;\frac{e^{-\frac{n}{2}r^2_n}}{\pi A^2(t)\sqrt{\sigma^2\sigma^2_n-c^2_n}},\label{eq: Ratio.Dist}
	\end{align}
	where
	\begin{align*}
		A(t)&=\sqrt{\frac{(t - c_n/\sigma^2)^2}{\sigma_n^2-c_n^2/\sigma^2} + \frac{1}{\sigma^2}},\\
		B(t)&=\left(\frac{\mu_n-c_n\mu/\sigma^2}{\sigma_n^2-c_n^2/\sigma^2}\right)(t - c_n/\sigma^2)\,+\, \frac{\mu}{\sigma^2},\\
		r^2_n&=\frac{(\mu_n-c_n\mu/\sigma^2)^2}{\sigma_n^2-c_n^2/\sigma^2} + \frac{\mu^2}{\sigma^2}.
	\end{align*}
\end{thm}
\begin{proof}
	For all $T\in\R$
	\begin{align*}
		F_{\Lambda_n}(T)&=\P[\Lambda_n\le T]\\
		&=\P[\widetilde{\Lambda}_n\le T-c_n/\sigma^2]\\
		&=\int_{-\infty}^{T-c_n/\sigma^2}f_{\widetilde{\Lambda}_n}(u)\, du\\
		&=\int_{-\infty}^{T}f_{\widetilde{\Lambda}_n}(t-c_n/\sigma^2)\, dt,\\
	\end{align*}
	and so,
	\begin{equation*}
		f_{\Lambda_n}(t)
		=
		f_{\widetilde{\Lambda}_n}(t-c_n/\sigma^2),
	\end{equation*}
	where $f_{\widetilde{\Lambda}_n}$ is given by Lemma \ref{lmm: offset}. Now, one may note
	\[
	\frac{(t - c_n/\sigma^2 - \tilde\mu_n/\mu)^2}{(t-c_n/\sigma^2)^2 + \tilde\sigma^2_n/\sigma^2}
	=
	\frac{(t - \mu_n/\mu)^2}{(t - \sigma_n/\sigma)^2 + 2t(1-\rho_n)\sigma_n/\sigma}.
	\]
\end{proof}
There is another approach to investigate the ratio distribution provided by Katz (1978) distribution approximation \cite{katz1978obtaining}. Here, we formulate it for $\Lambda_n$ by the following proposition.
\begin{pro}\label{pro: logdist}
	For $\mu\neq 0$ and $n\to\infty$, the process $\Lambda_n$ admits the approximatly logarithmic Gaussian distribution
	\begin{equation}\label{eq: logdist}
		\Lambda_n
		\sim\frac{\mu_n}{\mu}
		\cdot\LN\Bigg(0,
		\frac{\frac{\sigma^2}{\mu^2}
			+ \frac{\sigma_n^2}{\mu_n^2}
			- \frac{2c_n}{\mu_n\mu}
		}{n}
		\Bigg).
	\end{equation} 
\end{pro}
\begin{proof}
	From \eqref{eq: Tfrac1} we have
	\begin{equation}\label{eq: Tfrac2}
		\Lambda_n=\frac{\mu_n}{\mu}\cdot\frac{1 + u_n/\mu_n}{1 + z_n/\mu},
	\end{equation}
	where $u_n\sim\mathcal{N}(0,\sigma_n^2/n)$ and $z_n\sim\mathcal{N}(0,\sigma^2/n)$. Taking the logarithm of \eqref{eq: Tfrac2}, we have
	\begin{equation}\label{eq: TfracLog}
		\log\Lambda_n = \log\left(\frac{\mu_n}{\mu}\right)
		+ \log\left(1 + \frac{u_n}{\mu_n}\right)
		- \log\left(1 + \frac{z_n}{\mu}\right).
	\end{equation}
	Here, we apply the logarithmic power series, covergent on $|x|<1$, 
	$$
	\log(1+x)=\sum_{k=0}^\infty(-1)^k\frac{x^{k+1}}{k+1}=x-\frac{x^2}{2}+\frac{x^3}{3}-\cdots
	$$
	to approximate the two final part of the right hand side of \eqref{eq: TfracLog}. By the first power $k=1$, for $n\to\infty$ we have
	\begin{align}
		\log\Lambda_n 
		&\approx\log\left(\frac{\mu_n}{\mu}\right)
		+ \frac{u_n}{\mu_n}
		- \frac{z_n}{\mu}\notag\\
		&\sim\log\left(\frac{\mu_n}{\mu}\right)
		+\mathcal{N}\Bigg(0,
		\frac{\frac{\sigma^2}{\mu^2}
			+ \frac{\sigma_n^2}{\mu_n^2}
			- \frac{2c_n}{\mu_n\mu}
		}{n}\label{eq: Tlogapprox}
		\Bigg),
	\end{align}
	or equivalently
	\begin{equation*}
		\Lambda_n
		\approx\frac{\mu_n}{\mu}
		\cdot\exp\left(\frac{u_n}{\mu_n} - \frac{z_n}{\mu}\right),
	\end{equation*}
	and this proves \eqref{eq: logdist}.
\end{proof}
\begin{rem}
	The asymptotic ratio distributions \eqref{eq: Ratio.Dist} and \eqref{eq: logdist} are indeed usefull when the sample distribution is unknown. However, if the sample distribution is known, then by Theorem \ref{thm: LambdaConverg}, Theorem \ref{thm: asymp.norm}, and Lemma \ref{lmm: COV[u,z]} for $n\to\infty$ we can apply 
	\begin{align*}
		\mu_n &\approx a_{q_p},\\ 
		\sigma^2_n &\approx (b^+_{q_p})^2 + (b^-_{q_p})^2 + 2a^+_{q_p}a^-_{q_p},\\
		c_n& \approx b^2_{q_p}-\tilde a_{q_p}a_{q_p},
	\end{align*} 
	in \eqref{eq: Ratio.Dist}, and also \eqref{eq: logdist} returns 
	\begin{gather}\label{eq: LNdist}
		\Lambda_n
		\sim \frac{a_{q_p}}{\mu}
		\cdot\LN\left(0,
		\frac{
		\frac{\sigma^2}{\mu^2}
		+ \frac{(b^+_{q_p})^2 + (b^-_{q_p})^2 + 2a^+_{q_p}a^-_{q_p}}{a^2_{q_p}}
		- \frac{2b^2_{q_p}-2\tilde a_{q_p}a_{q_p}}{\mu a_{q_p}}}
		{n}
		\right).
	\end{gather}
\end{rem}
\section{Simulation and Callibration}\label{sec:simulate_callibrate}
In this section, we conduct Monte Carlo simulations of the distribution \eqref{eq: logdist} with $p=80\%$ for $n = 1000$ i.i.d variables, and $N = 10^5$ replications, with variables' common (continuous) distributions 
\begin{gather*}
	\mathtt{Normal(\mu,\sigma^2), Lognormal(\mu,\sigma^2), Exponential(\mu),}\\
	\mathtt{Raileigh(b), Generalized Pareto(k,s,\theta), Gamma(\alpha,\theta),}
\end{gather*}
where $\mu = \theta = 1, \sigma = k = b = s = 0.25, \alpha = 3$, and their practical estimated distribution. To enable a clear comparison between the log-normal formulated density function and the empirically estimated density function, both are plotted on the histogram of the $\Lambda_n$ statistics in Figure \ref{Fig:1}. One can easily observe how closely they approximate the actual density of this statistic, although, due to the sample size and the numbers $n$ and $N$, there are always some differences between the formulated and estimated distributions. The accumulated area between the two curves (which represents the difference in cumulative probability) is reported in Table \ref{Tbl:1}.
\begin{figure}[H]
	\includegraphics[scale = .8]{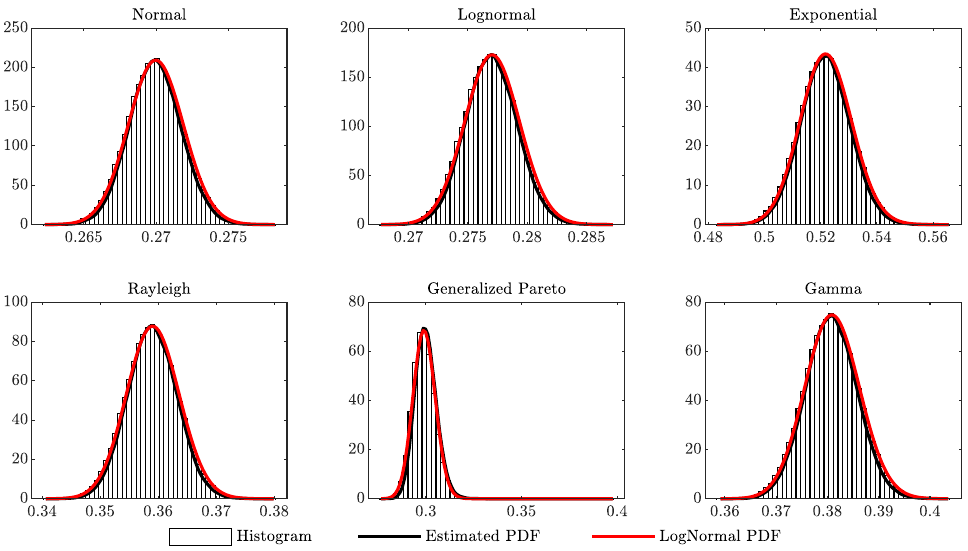}
	\caption{The histogram, analytical log-normal probability density function, and estimated probability density function of $\Lambda_n$ for i.i.d variables from different continuous distributions.}\label{Fig:1}
\end{figure}
\begin{table}[H]
	\centering
	\begin{tabular}{l|c}
		\hline
		\textbf{Distribution} & \textbf{Area between PDFs} \\
		\toprule
		Normal & 0.0713 \\
		LogNormal & 0.0708 \\
		Exponential & 0.0662 \\
		Rayleigh & 0.0687 \\
		Generalized Pareto & 0.0949 \\
		Gamma & 0.0709 \\
		\bottomrule
	\end{tabular}
	\caption{Cumulative area between analytic and estimated PDFs for different variable distributions.}\label{Tbl:1}
\end{table}
\section*{Acknowledgement}
The concept of quantile contributions was originally introduced to me by Tommi Sottinen and Klaus Grobys in the context of another study on applied statistics in economics, sincerely grateful to them for their valuable insights and inspiration. Nevertheless, the present work is an independent analytical study, and all parts of the research, analysis, and writing have been carried out solely by the author.
%
%%%%%%%%%%%%%%%%%%%%%%%%%%%%%%%%%%%%%%%%%%%%%%%%%%%%%%%%%%%%%%%%%%%%%%%%%%%%%%%
\bibliographystyle{siam}

\end{document}